\documentclass[a4paper,11pt,final]{article}
\usepackage{amsmath, amsthm, amssymb}
\usepackage{latexsym}
\usepackage[latin1]{inputenc}
\usepackage{epsf,exscale,times}
\usepackage[T1]{fontenc}
\usepackage[dvips]{graphicx}
\usepackage{graphics,epsfig}
\usepackage{epsfig,rotate,color}
\usepackage{showkeys}
\usepackage{subfigure}  

\newcommand{\SR}{\mathcal{S}(\mathbb{R})}

\newcommand{\I}{\mathcal{I}}

\newcommand{\F}{\mathcal{F}}
\newcommand{\R}{\mathbb{R}}

\newcommand{\n}{\vert \vert}
\newtheorem{theoreme}{Theorem}
\newtheorem{lemme}{Lemma}
\newtheorem{definition}{Definition}
\newtheorem{proposition}{Proposition}

\newtheorem{remark}{Remark}

\title{Global-in-time existence of perturbations around travelling-waves}

\author{ Afaf Bouharguane \footnote{Institut de Math\'ematiques et Mod\'elisation de Montpellier, UMR 5149 CNRS, Universit\'e Montpellier 2, Place Eug\`ene Bataillon, CC 051
34095 Montpellier, France. Email:\, {\sffamily bouharg@math.univ-montp2.fr}
} 
}
\date{\today}

\begin{document}
\maketitle

\begin{abstract}
We investigate a fractional diffusion/anti-diffusion equation proposed by Andrew C. Fowler to describe the dynamics of sand dunes sheared by a fluid flow. In this paper, we prove the  global-in-time well-posedness in the neighbourhood of travelling-waves solutions of the Fowler equation. \\

\end{abstract}



\section{Introduction}

The study of mechanisms that allow the formation of structures such as sand dunes and ripples at the bottom of a fluid flow plays a crucial role in the understanding of coastal dynamics. The modeling of these phenomena is particularly complex since we must not only solve the Navier-Stokes or Saint-Venant equations 
with equation for sediment transport, but also take into account the evolution of the bottom. Instead of solving the whole system fluid flow, free surface and free bottom, nonlocal models of fluid flow interacting with the bottom were introduced in \cite{fowler1,clagree}. Among these models, we are interested in the following nonlocal conservation law \cite{fowler1,fowler2}:
\begin{equation}
\begin{cases}
\partial_t u(t,x) + \partial_x\left(\frac{u^2}{2}\right) (t,x) + \I [u(t, \cdot)] (x)
- \partial_{xx}^2 u(t,x) = 0 & t \in (0,T), x \in \mathbb{R}, \\
u(0,x)=u_0(x) & x \in \mathbb{R},
\end{cases}
\label{fowlereqn2}
\end{equation}
where $T$ is any given positive time, $u=u(t,x)$ represents the dune height (see Fig. \ref{dune}) and $\I$
is a nonlocal operator defined as follows: for any Schwartz
function $\varphi \in \SR$ and any $x \in \mathbb{R}$,
\begin{equation}
\I [\varphi] (x) := \int_{0}^{+\infty} |\xi|^{-\frac{1}{3}}
\varphi''(x-\xi) d\xi . \label{nonlocalterm}
\end{equation}

\begin{figure}[h!]
	\centering
	\includegraphics[scale=0.5]{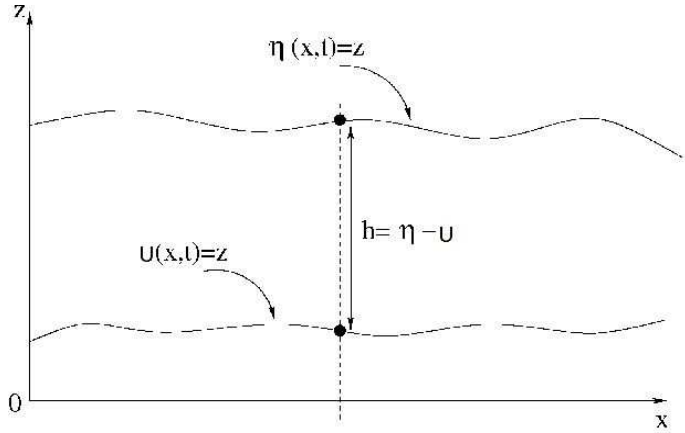} 
	\caption{Domain considered for the Fowler model: $h$ is the depth water, $\eta$ the free surface and $u$ the seabottom. }
	\label{dune}	
\end{figure}

Equation \eqref{fowlereqn2} is valid for a river flow over an erodible bottom $u(t,x)$ with slow variation and describes both accretion and erosion phenomena \cite{alibaud}. See \cite{alibaud,azaf} for numerical results on this equation.  \\
The nonlocal term $\I$ can be seen as a fractional power of order $2/3$ of the Laplacian with the bad sign. Indeed, it has been proved \cite{alibaud} 
\begin{equation} 
 \F \left( \I [\varphi]-\varphi''\right) (\xi) = \psi_{\mathcal{I}}(\xi)  \F
\varphi (\xi)
\label{fourier2}
\end{equation}
where 
\begin{equation}
\psi_{\mathcal{I}}(\xi)=4 \pi^2 \xi^2-a_{\I}
|\xi|^{\frac{4}{3}} + i \: b_{\I} \xi |\xi|^{\frac{1}{3}},
\label{psiexpression2}
\end{equation}
 with $a_\I, b_\I$ positive constants,
 $\F$ denotes the Fourier transform defined in \eqref{TF2} 
 and $\Gamma$ denotes the Euler function. One simple way to establish this fact is the derivation of a new formula for the operator $\I$, see Proposition \ref{intfor2}. 

\begin{remark} 
For causal functions (i.e. $\varphi(x) = 0 $ for $x<0$), this operator is 
 up to a multiplicative constant, the Riemann-Liouville fractional derivative operator which is defined as follows \cite{pod}
 \begin{equation}
\frac{1}{\Gamma(2/3)} \int_0^{+\infty} \frac{\varphi^{''}(x-\xi)}{|\xi|^{1/3}} d\xi= \frac{d^{-2/3}}{d x^{-2/3}} \varphi'' (x) = 
\frac{d^{4/3}}{d x^{4/3}} \varphi (x).
\end{equation}
\label{liouville0}
\end{remark} 

Therefore, the Fowler model has two antagonistic terms: a usual diffusion and a nonlocal fractional anti-diffusive term of lower order. This remarkable feature enabled to apply this model for signal processing. Indeed, the diffusion is used to reduce the noise whereas the nonlocal anti-diffusion  is used to enhance the contrast \cite{signal}. \\

Recently, some results regarding this equation have been obtained, namely, existence of travelling-waves $u_{\phi}(t,x) = \phi(x - ct)$ where $\phi \in C^1_{b}(\R)$ and $c \in \R$ represents wave velocity, the global well-posedness for $L^2$-initial data, the failure of the maximum principle and the local-in-time well-posedness in a subspace of $C^1_{b}$ \cite{alibaud, alvarez}.
 Notice that the travelling-waves are not necessarily of solitary type (see \cite{alvarez}) and therefore may not belong to $L^2(\R)$, the space where a global well-posedness result is available. In \cite{alvarez}, the authors prove local well-posedness in a subspace of  $C^1_{b}(\R)$ but fail to obtain global existence. \\
An interesting topic is to know if the shape of this travelling-wave is maintained when it is perturbed. This raises the question of the stability of travelling-waves. But before interesting ourselves to this problem, we have to show first
the global existence of perturbations around these travelling-waves. Hence in this paper,    
 we prove the global well-posedness in an $L^2$-neighbourhood of a regular travelling-wave, namely $u = u_{\phi} +v$. To prove this result, 
we consider the following Cauchy problem:
\begin{equation}
\begin{cases}
\partial_t v(t,x) + \partial_x(\frac{v^2}{2}+ u_{\phi} v)(t,x) + \I[v(t, \cdot)](x) -
\partial^2_{xx} v(t,x) = 0 & t \in (0,T), x \in \R, \\
v(0,x)=v_{0}(x) & x \in \R,
\end{cases}
\label{fowlermodif2}
\end{equation}
where $v_0 \in L^{2}(\R)$ is an initial perturbation and $T$ is any given positive time. \\

To prove the existence and uniqueness results, we begin by introducing the notion of \emph{mild} solution (see Definition \ref{mild2}) based on Duhamel's formula \eqref{duhamel2}, in which the kernel $K$ of $\I-\partial_{xx}^2$ appears. The use of this formula allows to prove the local-in-time existence with the help of a contracting fixed point theorem. The global existence is obtained thanks to an energy estimate \eqref{energy2}. 
This approach is classical: we refer for instance to \cite{alibaud, JD1}. \\

The plan of this paper is organised as follows. In the next section, we define the notion of mild solution to \eqref{fowlermodif2} and we give some properties on the kernel $K$ of $\I-\partial_{xx}^{2}$ that will be needed in the sequel.              
Sections \ref{uniqueness} and \ref{existence} are, respectively, devoted to the proof of the uniqueness and the existence of a mild solution for \eqref{fowlermodif2}. Section \ref{regularity} contains the proof of the regularity
of the solution. \\

\noindent \textbf{Notations.}\\
- The norm of a measurable function $f \in L^{p}(\R)$ is written $||f ||^p_{L^{p}(\R)} = \int_{\R} |f(x)|^{p} \, dx $ for $1 \leq p < \infty$. \\
- We denote by $\F$ the Fourier transform of $f$ which is defined by: for all $\xi \in \R$
\begin{equation}
\label{TF2}
\F f (\xi):= \int_{\R} e^{-2i\pi x \xi} f(x) dx,
\end{equation}
and $\F^{-1}$ denotes the inverse of Fourier transform. \\
- The Schwartz space of rapidly decreasing functions on $\R$ is denoted by $\SR .$ \\
- We write $C^{k}(\R) = \lbrace f: \R \rightarrow \mathbb{C}; f, f', \cdots, f^{(k) } \mbox{ are continous on } \R \rbrace $.\\
- We denote by $C_b(\R)$ the space of all bounded continuous real-valued functions on $\R$ with the norm $||.||_{L^{\infty}} =  \sup_{\R} |f|$. \\
- We write for any $T > 0$,
$$ C^{1,2}\left( (0,T] \times \R \right):= \lbrace u \in C\left( (0,T] \times \R\right) ; \, \partial_t u, \partial_x u, \partial_{xx}^2 u \in  C\left( (0,T] \times \R ) \right\rbrace .   $$
- We denote by $\mathcal{D}(U)$ the space of test functions on $U$ and $\mathcal{D}'(U)$ denotes the distribution space. 

\bigskip

Here is our main result. \\

\begin{theoreme} Let $T > 0$ and $v_0 \in L^{2}(\R)$.  There exists a unique mild solution $ v \in L^{\infty}\left( (0,T); L^{2}(\R)\right)$ of \eqref{fowlermodif2} (see Definition \ref{mild2}) which satisfies
$$ v \in C\left( [0, T] ; L^{2}(\R) \right)  \mbox{ and } v(0, \cdot) = v_0 \mbox{ almost everywhere.}$$
Moreover, if $\phi \in C^2_b(\R)$ then $ v \in C^{1,2}\left( (0, T] \times \R \right)$ and satisfies 
$$\partial_t v+ \partial_{x}\left( \frac{v^2}{2} + u_{\phi} v \right) + \I[v] - \partial_{xx}^2 v = 0, $$ 
on $(0,T]  \times \R$, in the classical sense or equivalently, $u = u_\phi + v$ is a classical solution of equation \eqref{fowlereqn2}.  \\
\label{theoreme1}
\end{theoreme}

\section{Duhamel formula and main properties of $K$}

\begin{definition}
\label{mild2}
Let $T>0$ and $v_{0} \in L^{2}(\R)$. We say that $v \in L^\infty((0,T);L^2(\R)) $ is a \textbf{mild solution} to \eqref{fowlermodif2} if for any $t \in \left(0,T\right)$: 
\begin{equation}
v(t,\cdot)=K(t, \cdot)\ast v_{0}- \int^{t}_{0} \partial_{x}K(t-s,\cdot)\ast \left(\frac{v^{2}}{2}+ u_{\phi}v \right)(s,\cdot) \: ds,
\label{duhamel2}
\end{equation}
where $K(t,x)=\F^{-1}\left(e^{-t \psi_{\I}(\cdot)}\right)(x)$
is the kernel of the operator $\I-\partial^2_{xx}$ and $\psi_\I$ is defined in \eqref{psiexpression2}.
\end{definition}

The expression \eqref{duhamel2} is the Duhamel formula and is obtained using the spatial Fourier transform.

\begin{proposition}[Main properties of $K$, \cite{alibaud}] \label{kernel20}
The kernel $K$ satisfies:
\begin{enumerate}
	\item $\forall t>0$, $K(t,\cdot) \in L^{1}\left(\R\right)$ and $K \in C^{\infty}\left((0,\infty)\times \R\right)$,
	\item $\forall s,t >0, \, \, K(s,\cdot)\ast K(t,\cdot)=K(s+t,\cdot)$,\\
	      $\forall u_{0} \in L^2\left(\R\right), \, \, \lim_{t \to 0} K\left(t,\cdot \right)*u_{0}=u_{0} \hspace{0.2 cm} in \hspace{0.2 cm} L^2\left(\R\right)$,
\item $ \forall T>0, \exists K_{0} \hspace{0.2 cm} such \hspace{0.2 cm} that$
$ \forall t \in \left(0,T\right], \hspace{0.3 cm} \vert \vert \partial_{x}K\left(t, \cdot \right)\vert \vert_{L^2\left(\R\right)}\leq K_{0} t^{-3/4}$,
\item $ \forall T>0, \exists K_{1} \hspace{0.2 cm} such \hspace{0.2 cm} that$
$\forall t \in \left(0,T\right], \hspace{0.3 cm} \vert \vert \partial_{x}K\left(t,\cdot \right)\vert \vert_{L^1\left(\R\right)}\leq K_{1} t^{-1/2}$.
\end{enumerate}
\end{proposition}

\begin{figure}[h!]
	\centering
	\includegraphics[width=12cm,height=68mm]{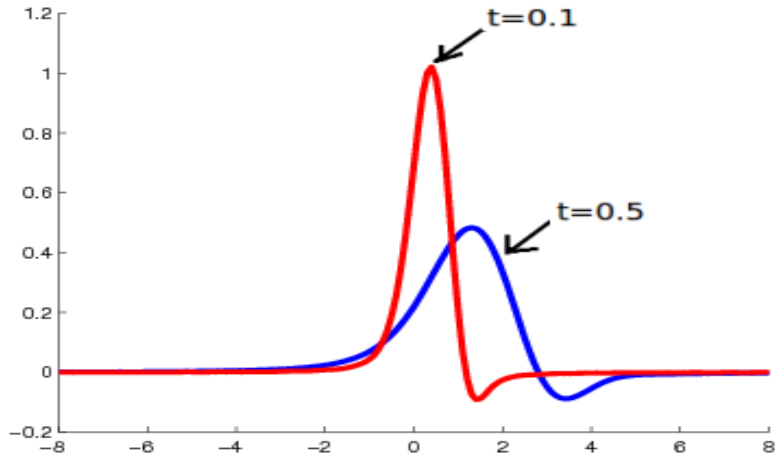} 
	\caption{Evolution of the kernel K for $t=0.1$ (red) and $t=0.5$ s (blue) }
	\label{kernelfig20}
\end{figure}

\begin{remark}
An interesting property for the kernel $K$ is the non-positivity (see Figure \ref{kernelfig20}) and the main consequence of this feature is the failure of maximum principle \cite{alibaud}. We use again this property to show that the constant solutions of the Fowler equation are unstable \cite{afaf}.

\end{remark}

\begin{remark}
\label{remarkestimation}
Using Plancherel formula, we have for any $v_0 \in L^{2}(\R)$ and any $t \in (0,T]$
\begin{equation*}
||K(t, \cdot) \ast v_{0} ||_{L^{2}(\R)} \leq e^{\alpha_0 t} ||v_{0} ||_{L^{2}(\R)},
\end{equation*}
where $\alpha_0 = - \min \mbox{Re}(\psi_{\I}) >0 $.
\end{remark}

\begin{proposition}[Integral formula for $\I$ \label{intfor2}]
For all $\varphi \in \SR$ and all $x \in \R$,
\begin{equation}
\I[\varphi](x) = \frac{4}{9} \int_{-\infty}^{0} \frac{\varphi(x+z) - \varphi(x) - \varphi'(x)z}{|z|^{7/3}} \, dz.
\label{integralformula}
\end{equation}
\end{proposition}

\begin{proof}
The proof is based on simple integrating by parts. The regularity and the rapidly decreasing of  $\varphi$ ensure the validity of the computations that follow. We have
\begin{eqnarray*}
\int_{0}^{+ \infty} \varphi''(x-\xi) |\xi |^{-1/3} \, d\xi &=& \int_{0}^{+ \infty} \frac{d}{d\xi} \left( \varphi'(x) -\varphi'(x-\xi) \right)  |\xi |^{-1/3} \, d\xi , \\
& = & \frac{1}{3} \int_{0}^{+\infty} |\xi |^{-4/3} \left( \varphi'(x) -\varphi'(x-\xi) \right) \, d\xi, \\
& = & \frac{1}{3} \int_{0}^{+\infty} |\xi |^{-4/3} \frac{d }{d\xi} \left( \varphi'(x) \xi + \varphi(x-\xi)-\varphi(x) \right) \, d\xi , \\
&=& \frac{4}{9} \int_{0}^{+ \infty} \frac{ \varphi(x-\xi)-\varphi(x) + \varphi'(x)\xi }{ |\xi |^{7/3} } \, d\xi , \\
&=& \frac{4}{9} \int_{-\infty}^{0} \frac{ \varphi(x + \xi)-\varphi(x) - \varphi'(x)\xi }{ |\xi |^{7/3} } \, d\xi .
\end{eqnarray*}
There is no boundary term at infinity (resp. at zero) because $\varphi$ is a rapidly decreasing function on $\R$ (resp. $\varphi$ is smooth). 
\end{proof}

\begin{remark}
Using integral formula \eqref{integralformula}, \cite{alibaud,alvarez} proved that 
\begin{equation*}
\F \left( \I[\varphi] \right)(\xi) = 4 \pi^2 \Gamma(\frac{2}{3})|\xi| ^{4/3}  \left( -\frac{1}{2}+ i \frac{\sqrt{3}}{2} \, sgn (\xi) \right) \F \varphi(\xi).
\end{equation*}
Notice that $ \F \left( \I[\varphi] \right)(\xi) = 4 \pi^2 \Gamma(\frac{2}{3}) (i \xi) ^{4/3} $ which is consistent with Remark \ref{liouville0}: up to a multiplicative constant $\I[\varphi]$ is $\bigskip \frac{d^{4/3}\varphi}{dx^{4/3}}$.
\label{fourier0}
\end{remark}

\begin{proposition}
Let $s\in \R$ and $\varphi \in  H^s(\R)$. Then $\I[\varphi] \in H^{s-4/3}(\R)$ and we have 
\begin{equation}
||\I[\varphi]||_{H^{s-4/3}(\R)} \leq 4 \pi^2 \Gamma(\frac{2}{3}) ||\varphi||_{H^s(\R)}.
\end{equation}
\label{estimnonlocal}
\end{proposition}

\begin{proof}
For all $s\in \R$ and all $\varphi \in  H^s(\R)$, we have, using Remark \ref{fourier0}
\begin{eqnarray*}
||\I[\varphi]||_{H^{s-4/3}(\R)} &=& \left( \int_\R (1 + |\xi|^2)^{s-4/3} |\F(\I[\varphi])(\xi)|^2 \, d\xi   \right)^{1/2}, \\
&=& 4 \pi^2 \Gamma(\frac{2}{3})  \left( \int_\R (1 + |\xi|^2)^{s-4/3} |\frac{1}{2}-i \mbox{ sgn}(\xi) \frac{\sqrt{3}}{2}| |\xi |^{8/3} |\F(\varphi)(\xi)|^2 \, d\xi \right)^{1/2} , \\
&=& 4 \pi^2 \Gamma(\frac{2}{3}) \left( \int_\R \left( \frac{|\xi|^2}{ 1+ |\xi|^2}\right) ^{4/3} (1 + |\xi|^2)^s | \F(\varphi)(\xi)|^2 \, d\xi \right)^{1/2}, \\
&\leq & 4 \pi^2 \Gamma(\frac{2}{3}) \left[ \int_\R (1 + |\xi|^2)^s | \F(\varphi)(\xi)|^2 \, d\xi \right]^{1/2}, \\
&=& 4 \pi^2 \Gamma(\frac{2}{3}) ||\varphi||_{H^s(\R)}.
\end{eqnarray*}
\end{proof}

\begin{remark} From the previous Proposition, we deduce that for all 
$s\in \R$ and all $\varphi \in  H^s(\R)$, $\I[\varphi] \in H^{s-4/3}(\R)$. In particular, using the Sobolev embedding $H^{2/3}(\R) \hookrightarrow C_b(\R) \cap L^{2}(\R)$, we deduce that $\I: H^2(\R) \rightarrow C_b(\R) \cap L^{2}(\R) $ is a bounded linear operator.
\label{remarkfourier}
\end{remark}

\begin{proposition}[Duhamel formula \eqref{duhamel2} is well-defined]\label{welldefined}
Let $T>0$, $v_{0} \in L^{2}(\R)$ and  $w \in  L^\infty((0,T);L^{1}(\R))  +  L^\infty((0,T);L^{2}(\R)) $. Then, the function 
\begin{equation}
v: t \in (0,T] \rightarrow K(t,\cdot) \ast v_0 - \int_{0}^{t}
\partial_x K (t-s,\cdot) \ast w(s,\cdot)ds
\end{equation}
is well-defined and belongs to $C([0,T];L^2(\R))$ ( being extended at $t = 0$ by the value $v(0, \cdot) = v_0$ ). 
\end{proposition}

\begin{proof}[Proof]
From Proposition \ref{kernel20}, it easy to see that $v$ is well-defined and that for any $t \in (0,T]$, $v(t, \cdot) \in L^{2}(\R)$. Indeed, $\forall t>0, \partial_{x}K(t,\cdot) \in L^{1}(\R)\cap L^{2}(\R)$ so by Young inequalities $\partial_{x}K(t,\cdot) \ast w(t,\cdot)$ exists and using the estimates on the gradient (item 3 and 4 of Proposition \ref{kernel20}) we deduce that $v$ is well-defined and $v(t,\cdot) \in L^2(\R)$. \\ 

Let us prove the continuity of $v$. By the second item of Proposition \ref{kernel20}, we have that the function $t \in (0,T]\rightarrow K(t,\cdot) \ast v_{0}$ is continuous and it is extended continuously up to $t=0$ by the value $v(0,\cdot)=v_0$.
We define the function
$$F:t \in [0,T]\rightarrow \int_{0}^{t}
\partial_x K (t-s,\cdot) \ast w(s,\cdot)ds.$$
Now, we are going to prove that $F$ is uniformly continuous. For any $h > 0$, Young inequalities imply
\begin{eqnarray} 
||F(t+h,\cdot)-F(t,\cdot)||_{L^2} 
& \leq & \int_{0}^{t} \n \partial_{x}K(t+h-s,\cdot)-\partial_{x}K(t-s,\cdot) \n_{L^{i}} \, ds \, \n w \n_{L^\infty((0,T);L^{j})}   \nonumber \\  
& + & \int_{t}^{t+h} \n \partial_{x}K(t+h-s,\cdot) \n_{L^{i}} \, ds \, \n w \n_{L^\infty((0,T);L^{j})}, \label{inegalite1}  
\end{eqnarray}
where $i,j \in \mathbb{N}^{*}$ are such that $i+j=3$. Since $\partial_{x} K(t, \cdot)=\F^{-1}(\xi \rightarrow 2 i \pi \xi e^{-t \psi_{\I}(\xi)})$, the dominated convergence theorem implies that
$$\n \partial_{x}K(t-s+h, \cdot)-\partial_{x}K(t-s,\cdot) \n_{L^{i}(\R)}\rightarrow 0, \mbox{ as } h\rightarrow 0.$$
Moreover, using the estimates on the gradient (item 3 and 4 of Proposition \ref{kernel20}), we have the following inequality
$$\int_{t}^{t+h}\n \partial_{x} K(t-s+h,\cdot) \n_{L^{j}(\R)} ds \leq c_{j} h^{\alpha_j}, $$
where $c_{j}$ is a positive constant and $\alpha_j = \left\{
    \begin{array}{ll}
        1/2 & \mbox{if } j=1 \\
        1/4 & \mbox{if } j=2 
    \end{array}
\right.
$.

From \eqref{inegalite1}, we obtain that $||F(t+h,\cdot)-F(t,\cdot)||_{L^2(\R)} \rightarrow 0, \mbox{ as } h\rightarrow 0$. Hence, the function $F$ is continuous and this completes the proof of the continuity of $v$.  
\end{proof}

\begin{remark} Using the semi-group property of the kernel $K$, we have for all $t_0 \in (0,T)$ and all $t \in [0, T-t_0]$,  \cite{alibaud} 
\begin{equation*}
v(t + t_0, \cdot) = K(t,\cdot) \ast v(t_0, \cdot) - \int_{0}^{t}
\partial_x K (t-s,\cdot) \ast w(t_0 + s,\cdot) \, ds.
\end{equation*}

\label{semigroupe}
\end{remark}

\section{Uniqueness of a solution \label{uniqueness}}

Let us first establish the following Lemma.

\begin{lemme}
Let $T>0$ and $v_0 \in L^2(\R)$. For $i =1, 2$, let $w_i \in  L^\infty((0,T);L^{1}(\R)) \cup L^\infty((0,T);L^{2}(\R)) $ and define $v_i$ as in Proposition \ref{welldefined} by:
\begin{equation*}
v_i(t, \cdot) = K(t,\cdot) \ast v_0 - \int_{0}^{t}
\partial_x K (t-s,\cdot) \ast w_i(s,\cdot) \, ds.
\end{equation*}
Then, 
\begin{equation*}
||v_1 - v_2 ||_{C([0,T]; L^{2}(\R))} \leq \left\{
    \begin{array}{ll}
        4 K_0 T^{1/4} ||w_1 - w_2 ||_{L^\infty((0,T);L^{1}(\R))}  & \mbox{if } w_i \in L^\infty((0,T);L^{1}(\R)), \\
        2 K_1 \sqrt{T} \, ||w_1 - w_2 ||_{L^\infty((0,T);L^{2}(\R))} & \mbox{if } w_i \in L^\infty((0,T);L^{2}(\R)).
    \end{array}
\right.
\end{equation*}
\end{lemme}

\begin{proof}
For all $t \in [0,T]$, we have
\begin{equation*}
v_1(t, \cdot) - v_2(t, \cdot) = - \int_{0}^{t} \partial_x K(t-s, \cdot) \ast (w_1 - w_2)(s, \cdot) \, ds.
\end{equation*}
Hence with the help of Young inequalities, we get
\begin{equation*}
||v_1(t, \cdot) - v_2(t, \cdot) ||_{L^{2}(\R)} \leq \left\{
    \begin{array}{ll}
        \int_{0}^{t} ||\partial_x K(t-s, \cdot) ||_{L^{2}(\R)} ||(w_1 - w_2)(s, \cdot)||_{L^{1}(\R)} \, ds \\ 
         \hspace{6 cm} \mbox{if } w_i \in L^\infty((0,T);L^{1}(\R)), \\
         \int_{0}^{t} ||\partial_x K(t-s, \cdot) ||_{L^{1}(\R)} ||(w_1 - w_2)(s, \cdot)||_{L^{2}(\R)} \, ds \\ 
        \hspace{6 cm} \mbox{if } w_i \in L^\infty((0,T);L^{2}(\R)).
    \end{array}
\right.
\end{equation*}
It then follows that 
\begin{equation*}
||v_1(t, \cdot) - v_2(t, \cdot) ||_{L^{2}(\R)} \leq \left\{
    \begin{array}{ll}
        \int_{0}^{t} ||\partial_x K(t-s, \cdot) ||_{L^{2}(\R)} \, ds \, ||w_1 - w_2||_{L^\infty((0,T);L^{1}(\R))} \\ 
         \hspace{6 cm} \mbox{if } w_i \in L^\infty((0,T);L^{1}(\R)), \\
        \int_{0}^{t} ||\partial_x K(t-s, \cdot) ||_{L^{1}(\R)} \, ds \, ||w_1 - w_2||_{L^\infty((0,T);L^{2}(\R))} \\               \hspace{6 cm} \mbox{if } w_i \in L^\infty((0,T);L^{2}(\R)).
    \end{array}
\right.
\end{equation*}
Using again the estimates of the gradient of $K$ (see Proposition \ref{kernel20}), we conclude the proof of this Lemma. 
\end{proof}

\begin{proposition}
Let $T>0$ and $ v_0 \in L^2(\R)$. There exists at most one $v \in L^\infty((0,T);L^{2}(\R))$ which is a mild solution to \eqref{fowlermodif2}. 
\end{proposition}

\begin{proof}
Let $v_1, v_2 \in L^\infty((0,T);L^{2}(\R))$ be two mild solutions to \eqref{fowlermodif2} and $t \in [0,T]$.
Using the previous Lemma, we get
\begin{equation*}
||v_1 - v_2 ||_{C([0,t]; L^{2}(\R))} \leq 2 K_0 t^{1/4} ||v_1^2 - v_2^2 ||_{L^\infty((0,t);L^{1}(\R))} + 2 K_1 \sqrt{t} ||u_\phi v_1 -u_\phi v_2 ||_{L^\infty((0,t);L^{2}(\R))}.
\end{equation*}  
Since,
\begin{equation}
||v_1^2 - v_2^2 ||_{L^\infty((0,t);L^{1}(\R))} \leq M ||v_1 - v_2 ||_{C([0,t]; L^{2}(\R))}
\label{Mmax}
\end{equation}
with $M = ||v_1||_{C([0,T]; L^{2}(\R))} + ||v_2||_{C([0,T]; L^{2}(\R))}, $ \\
then
\begin{equation*}
||v_1 - v_2 ||_{C([0,t]; L^{2}(\R))} \leq (2 M K_0 t^{1/4} + 2 K_1 t^{1/2}||u_\phi ||_{C^1_b(\R)})||v_1 - v_2 ||_{C([0,t]; L^{2}(\R))}  .
\end{equation*} 
Therefore, $v_1 = v_2$ on $[0,t]$ for any $t \in (0,T]$ satisfying $2 M K_0 t^{1/4} + 2 K_1 t^{1/2}||u_\phi ||_{C^1_b(\R)} < 1 $.  
Since $v_1$ and $v_2$ are continuous with values in $L^2(\R)$, we have that $ v_1 = v_2 $ on $[0, T_{*}]$ where $T_{*}$ is the positive solution of the following equation \\ 
$$2 M K_0 t^{1/4} + 2 K_1 t^{1/2}||u_\phi ||_{C^1_b(\R)} = 1, $$ 
i.e. $T_{*} = ( \frac{-2M K_0 + \sqrt{ 4 M^2 K_0^2 + 8 K_1 ||u_\phi||_{C^{1}_b(\R)}}} {4 K_1 ||u_\phi||_{C^{1}_b(\R) }} )^{4}. $  \\
To prove that $v_1 = v_2 $ on  $[0,T]$, let us define 
\begin{equation*}
t_0 := \sup \lbrace t \in [0, T] \mbox{ s.t } v_1 = v_2 \hspace{0.1 cm} [0, t] \rbrace
\end{equation*}
and we assume that $t_0 < T$. By continuity of $v_1$ and $v_2$, we have that $v_1(t_0, \cdot) = v_2(t_0, \cdot) $. 
Using the semi-group property, see Remark \ref{semigroupe}, we deduce that $v_1(t_0 + \cdot, \cdot) = v_2(t_0 + \cdot, \cdot) $ are mild solutions to \eqref{fowlermodif2} with the same initial data $ v_1(t_0, \cdot) = v_2(t_0, \cdot)$ which implies, from the first step of this proof, that $v_1(t, \cdot) = v_2(t, \cdot) $ for $t \in [t_0, T_{*} +t_0]$. Finally, we get a contradiction with the definition of $t_0$ and we infer that $t_0 = T$. This completes the proof of this proposition. 
\end{proof}

\section{Global-in-time existence of a mild solution \label{existence}}

\begin{proposition}[local-in-time existence]  \label{existlocal}
Let $v_0 \in L^2(\R)$. There exists $T_\ast>0$ that only depends on
$||v_0||_{L^2(\R)}$ and $\n u_\phi \n_{C_{b}^{1}(\R)}$ such that \eqref{fowlermodif2} admits a unique
mild solution $v \in C([0,T_\ast]; L^{2}(\R)) \cap C((0,T_\ast]; H^{1}(\R)) .$ 
Moreover, $v$ satisfies
\begin{equation*}
\sup_{t \in (0,T_*]} t^{1/2} ||\partial_x v(t, \cdot)||_{L^2(\R)} < + \infty.
\end{equation*}
\end{proposition}

\begin{proof}
For $v \in C([0,T]; L^{2}(\R)) \cap C((0,T]; H^{1}(\R))$, we consider the following norm
\begin{equation}
\n \vert v \n \vert :=||v||_{C([0,T];L^2(\R))}+\sup_{t \in (0,T]}
t^\frac{1}{2} ||\partial_{x} v(t,\cdot)||_{L^2(\R)}
\label{norme}
\end{equation}   
and we define the affine space 
$$
X:= \left\{v \in C([0,T];L^2(\R)) \cap C((0,T];H^1(\R))
\mbox{ s.t. $v(0,\cdot)=v_0$ and $|||v||| <+\infty$} \right\}.
$$
It is readily seen that $X$ endowed with the distance induced by the norm $|||\cdot|||$ is a complete metric space. For $v \in X$, we define the function
\begin{equation*}
\Theta v: t \in [0,T] \rightarrow K(t,\cdot) \ast v_0 -
\frac{1}{2} \int_{0}^{t}
\partial_x K (t-s,\cdot) \ast v^2(s,\cdot) \, ds -\int_{0}^{t}
\partial_x K (t-s,\cdot) \ast u_\phi v(s,\cdot) \, ds.
\end{equation*}
From Proposition \ref{welldefined}, $\Theta v \in C([0,T]; L^2(\R))$ and satisfies $\Theta v(0,\cdot)=v_{0}$.\\ \\
\noindent \textbf{First step: $\Theta v \in X$.}
Since 
$$
\partial_{x} (K(t,\cdot) \ast v_0)=\partial_x K(t, \cdot) \ast v_0=\F^{-1}(\xi \mapsto 2i \pi \xi e^{-t \psi_{\I}(\xi)} \F v_0(\xi)),
$$
the dominated convergence theorem implies that for any $t_0 >0$, 
$$
\int_\R  4 \pi^2 |\xi|^2 \left|e^{-t \psi_\I(\xi)}-e^{-t_0
\psi_\I(\xi)}\right|^2 | \F v_0(\xi)|^2 d\xi \rightarrow 0, \quad
\mbox{as $t \rightarrow t_0$}.
$$
Therefore, the function $t>0 \rightarrow (\xi \mapsto 2i \pi \xi e^{-t \psi_{\I}(\xi)} \F v_0(\xi)) \in L^{2}(\R)$ is continuous and since $\F$ is an isometry of $L^2$, we deduce that $t>0 \rightarrow \partial_x K(t,\cdot) \ast v_0 \in L^2(\R)$ is continuous. We have then established that $t>0 \rightarrow K(t,\cdot) \ast v_0 \in H^1(\R)$ is continuous. Moreover, from Proposition \ref{kernel20}, we have
\begin{equation}
||\partial_x K(t, \cdot) \ast v_0 ||_{L^2(\R)} \leq K_1 t^{-1/2} ||v_0||_{L^2(\R)}.
\label{gradientestimate}
\end{equation}
Let $w$ denote the function
\begin{equation*}
w(t,\cdot)=\frac{1}{2} \int_{0}^{t}
\partial_x K (t-s,\cdot) \ast v^2(s,\cdot)ds +\int_{0}^{t}
\partial_x K (t-s,\cdot) \ast u_{\phi} v(s,\cdot)ds.
\end{equation*}
Let us now prove that $w \in C((0,T]; H^{1}(\R))$. We first have
$$
\partial_{x} w(t, \cdot)=\int_{0}^{t} \partial_x K (t-s,\cdot) \ast v \partial_{x}v(s,\cdot)ds + \int_{0}^{t}
\partial_x K (t-s,\cdot) \ast \partial_{x}(u_{\phi} v)(s,\cdot)ds .
$$
Using Young inequalities and Proposition \ref{kernel20}, we get
\begin{eqnarray*} 
\n \partial_{x} w(t,\cdot) \n_{L^{2}(\R)} & \leq & \int_{0}^{t} \n \partial_x K (t-s,\cdot)\ast v \partial_{x}v(s,\cdot) \n_{L^{2}(\R)} ds \nonumber \\
&\,&  \hspace{2 cm} +  \int_{0}^{t} \n \partial_x K (t-s,\cdot)\ast \partial_{x}(u_\phi v)(s,\cdot)\n_{L^{2}(\R)} ds,  \nonumber \\
&\leq &  \int_{0}^{t} \n \partial_x K (t-s,\cdot)\n_{L^{2}(\R)} \n v \partial_{x}v(s,\cdot) \n_{L^{1}(\R)} ds  \nonumber  \\ 
& \, & \hspace{2 cm} + \int_{0}^{t} \n \partial_x K (t-s,\cdot) \n_{L^{1}(\R)} \n \partial_{x}(u_{\phi} v)(s,\cdot) \n_{L^{2}(\R)} ds, \nonumber \\
& \leq & \n v\n_{C \left( [0,T];L^{2}(\R) \right) } \int_{0}^{t} K_0 (t-s)^{-3/4}s^{-1/2} ds \sup_{s \in (0,T]} s^{1/2} \n \partial_{x} v(s,.)\n_{L^2(\R)} \nonumber \\
& \, & \hspace{2 cm} + \int_{0}^{t}K_1 (t-s)^{-1/2}s^{-1/2} ds \sup_{s \in (0,T]} s^{1/2} \n \partial_{x} (u_{\phi} v)(s,\cdot)\n_{L^2(\R)}. \nonumber \\
\end{eqnarray*}
We then obtain
\begin{eqnarray} \label{estim}
\n \partial_{x} w(t,\cdot) \n_{L^{2}(\R)} & \leq & K_0 I \n v\n_{C([0,T];L^{2}(\R))} T^{-1/4} \sup_{s \in (0,T]} s^{1/2} \n \partial_{x} v(s,\cdot)\n_{L^2(\R)} \nonumber \\
& \, & \hspace{2 cm} + K_1 J \sup_{s \in (0,T]} s^{1/2} \n \partial_{x}(u_{\phi} v)(s,\cdot) \n_{L^{2}(\R)}, 
\end{eqnarray}
where $I=B(\frac{1}{2},\frac{1}{4})$ and $J=B(\frac{1}{2},\frac{1}{2})=\pi$, $B$ being the beta function defined by
\begin{equation*}
B(x,y):= \int_{0}^{1} t^{x-1} (1-t)^{y-1} dt.
\end{equation*}
As $|||v|||< \infty$ then 
$$\sup_{s \in (0,T]} s^{1/2} \n \partial_{x} v(s,\cdot)\n_{L^2(\R)}< \infty \mbox{ and } \sup_{s \in (0,T]} s^{1/2} \n \partial_{x}(u_\phi v)(s,\cdot) \n_{L^{2}(\R)} < \infty.$$ 
We then deduce that $\partial_{x} w(t,\cdot)$ is in $L^2$ and so $\partial_{x} \Theta v(t,\cdot) \in L^{2}(\R)$ for all $t \in (0,T]$. \\

Let us now prove that $\partial_{x} w$ is continuous on $(0,T]$ with values in $L^2$. \\
For $\delta >0$ and $t \in (0,T]$, we define
\begin{eqnarray*}
(\partial_{x} w)_{\delta}(t,\cdot)&:=&\int_{0}^{t}
\partial_x K (t-s,\cdot) \ast 1_{ \lbrace s>\delta \rbrace}(v \partial_x v)(s,\cdot) \,ds \\ 
&\,& \hspace{1 cm} + \int_{0}^{t}
\partial_x K (t-s,\cdot) \ast 1_{ \lbrace s>\delta \rbrace}\partial_{x}(u_{\phi} v)(s,\cdot) \, ds.
\end{eqnarray*}
Since $1_{ \lbrace s>\delta \rbrace}(v \partial_x v)(s,\cdot) \in L^{\infty}([0,T]; L^{1}(\R))$ and $1_{ \lbrace s>\delta \rbrace }\partial_{x}(u_{\phi} v)(s,\cdot) \in L^{\infty}([0,T]; L^{2}(\R))$ then Proposition \ref{welldefined} implies that $(\partial_x w)_{\delta}: [0,T]\rightarrow L^{2}(\R)$ is continuous.
Moreover, we have for any $t \in (0,T] $ and $\delta \leq t$,
\begin{eqnarray*}
||\partial_x w(t,\cdot)-(\partial_x w)_\delta(t,\cdot)||_{L^2} & \leq & K_0 \int_{0}^{\delta} (t-s)^{-3/4}s^{-1/2} \, ds \n v \n_{C \left( [0,T];L^{2} \right) } \sup_{s \in (0,T]} s^{1/2} \n \partial_{x} v(s,\cdot) \n_{L^{2}} \\ 
&\,& \hspace{1 cm} + \, K_1 \int_{0}^{\delta} (t-s)^{-1/2}s^{-1/2}ds \sup_{s \in (0,T]} s^{1/2} \n \partial_{x} (u_{\phi} v)(s,\cdot) \n_{L^{2}}.
\end{eqnarray*}
It then follows that
\begin{equation*}
\sup_{t \in (0,T]} ||\partial_x w(t,\cdot)-(\partial_x w)_\delta(t,\cdot)||_{L^2(\R)}\rightarrow 0 \mbox{ as } \delta \rightarrow 0.
\end{equation*}
We next infer that $\partial_x w \in C((0,T];L^2(\R))$ because it is a local uniform limit of continuous functions. Hence, we have established that $\Theta v \in C([0,T]; L^{2}(\R)) \cap C((0,T]; H^{1}(\R)) $. To prove that $\Theta v \in X$, it remains to show that $|||\Theta v |||< +\infty$.  
Using \eqref{gradientestimate} and \eqref{estim}, we have
\begin{eqnarray} \label{estime}
\sup_{t \in (0,T]} t^{1/2} \n \partial_x \Theta v(t, \cdot) \n_{L^{2}} &\leq& K_1 \n v_0 \n_{L^{2}}+ K_0 I T^{1/4} 
\sup_{s \in (0,T]} s^{1/2} \n \partial_x v(s, \cdot) \n_{L^{2}} \n v \n_{C \left(  [0,T];L^{2} \right) } \nonumber \\
&+& K_1 J T^{1/2} \sup_{s \in (0,T]} s^{1/2} \n \partial_x (u_{\phi} v)(s, \cdot) \n_{L^{2}}.
\end{eqnarray}
Finally, we have $\Theta: X \longrightarrow X$.\\

\noindent \textbf{Second step:}
We begin by considering a ball of $X$ of radius $R$ centered at the origin 
$$
B_{R}:=\left\{v \in X  \, \, / \, \,  \vert \n v \vert \n \leq R \right\} 
$$ 
where $R> \n v_{0} \n_{L^{2}(\R)}+ K_1 \n v_0 \n_{L^{2}(\R)}$. Take $v \in B_{R}$ and let us now prove that $\Theta$ maps $B_{R}$ into itself. We have
$$ 
||\Theta v(t,\cdot)||_{L^2(\R)} \leq \n K(t,\cdot) \ast v_0 \n_{L^{2}(\R)}+  \int^{t}_{0} \n \partial_{x} K(t-s,\cdot)\ast \left(\frac{v^{2}}{2}+u_{\phi} v\right)(s,\cdot)\n_{L^{2}(\R)} \: ds . 
$$ 
By Remark \ref{remarkestimation}, we get
\begin{equation} \label{estimation}
\n K(t,\cdot) \ast v_0 \n_{L^{2}(\R)} \leq e^{\alpha_{0}T} \n v_{0} \n_{L^{2}(\R)},
\end{equation}
where $\alpha_{0}=-\min \mbox{Re}(\psi_{\I})>0$. Moreover, since $\n v^{2} \n_{L^\infty((0,T);L^1(\R))}=\n v \n^2_{L^\infty((0,T);L^{2}(\R))}$ and with the help of Proposition \ref{kernel20}, we get
\begin{eqnarray} \label{estima}
||\Theta v (t,\cdot)||_{L^2(\R)} & \leq & e^{\alpha_{0}T} \n v_{0} \n_{L^{2}(\R)}+ 2K_0 T^{1/4} R^{2} +2 K_1 T^{1/2} \n u_\phi \n_{C^{1}_{b}(\R)} R.
\end{eqnarray}
Using \eqref{estime} and \eqref{estima}, we deduce that 
\begin{eqnarray*}
|||\Theta v||| &\leq& e^{\alpha_{0}T} \n v_{0} \n_{L^{2}(\R)}+ K_1 \n v_{0} \n_{L^{2}(\R)}+ (2+ I)K_0 T^{1/4} R^2+(2+J)R K_1 T^{1/2} \n u_\phi \n_{C^{1}_{b}(\R)}  \\ 
&\,&  \hspace{0 cm} + K_1 J\n u_\phi \n_{C^{1}_{b}(\R)} R T.\\
\end{eqnarray*}
Therefore, for $T>0$ sufficiently small such that
\begin{eqnarray} \label{Tsmall}
&\,& e^{\alpha_{0}T} \n v_{0} \n_{L^{2}(\R)}+ K_1 \n v_{0} \n_{L^{2}(\R)}+ (2+ I)K_0 T^{1/4} R^2+(2+J)R K_1 T^{1/2} \n u_\phi \n_{C^{1}_{b}(\R)} \nonumber \\
&\,& \hspace{3 cm} + \, K_1 J\n u_\phi \n_{C^{1}_{b}(\R)} R T \leq R,
\end{eqnarray}
we get that $|||\Theta v||| \leq R$. \\
To finish with, we are going to prove that $\Theta$ is a contraction. \\
For $v,w \in B_{R}$, we have for any $t \in (0,T)$ 
\begin{eqnarray*} 
\n \Theta v (t, \cdot)-\Theta w (t, \cdot) \n_{L^{2}(\R)} & \leq &  \frac{1}{2} \int^{t}_{0} \n \partial_{x} K(t-s,\cdot)\n_{L^{2}(\R)} 
\n (v^2-w^2)(s,\cdot) \n_{L^1(\R)} ds \\ 
&\,& \hspace{1.5 cm} + \int^{t}_{0}  \n \partial_{x} K(t-s,\cdot)\n_{L^{1}(\R)} \n u_\phi (v-w)(s,\cdot) \n_{L^{2}(\R)} ds, \\
&\leq & 2 K_0 t^{1/4} \n v^2-w^2\n_{C([0,T];L^1(\R))} \\
&\, & \hspace{1.5 cm} + \, 2 K_1 t^{1/2} \n u_\phi \n_{C^{1}_{b}(\R)}\n v-w\n_{C([0,T];L^2(\R))}, 
\end{eqnarray*} 
and since,
\begin{eqnarray*}
\n v^2-w^2\n_{C([0,T];L^1(\R))} & \leq & (\n v\n_{C([0,T];L^2(\R))}+\n w\n_{C([0,T];L^2(\R))})\n v-w\n_{C([0,T];L^2(\R))}, \\
& \leq & 2 R \n v-w\n_{C([0,T];L^2(\R))},
\end{eqnarray*}
we get
\begin{equation} \label{estim2}
\n \Theta v(t, \cdot)-\Theta w(t, \cdot) \n_{L^{2}(\R)} \leq  (4R K_0 t^{1/4}+2K_1 t^{1/2} \n u_\phi \n_{C^{1}_{b}(\R)}) \n v-w\n_{C([0,T];L^2(\R))}.
\end{equation}
Moreover
\begin{eqnarray*}
\n \partial_{x}(\Theta v-\Theta w)(t,\cdot) \n_{L^2(\R)} & \leq & \frac{1}{2} \int^{t}_{0} \n \partial_{x} K(t-s,\cdot) \ast \partial_{x}(v^{2}-w^{2})(s,\cdot)\n_{L^{2}(\R)} ds \\ 
& \,& \hspace{1 cm} + \int^{t}_{0} \n \partial_{x} K(t-s,\cdot) \ast \partial_{x}(u_\phi (v-w))(s,\cdot)\n_{L^{2}(\R)} ds, \\
& \leq & K_0 I t^{-1/4} \sup_{s \in (0,T]} s^{1/2} \n (v \partial_x v-w\partial_x w)(s, \cdot) \n_{L^{1}(\R)} \\ 
&\,& \hspace{1 cm} + \, K_1 J \sup_{s \in (0,T]} s^{1/2} \n \partial_x \left( u_{\phi}(v -w) \right) (s,\cdot) \n_{L^{2}(\R)}.
\end{eqnarray*}
And since
\begin{equation*}
\n (v \partial_x v-w\partial_x w )(t,\cdot)\n_{L^{1}} \leq \n \partial_x w(t,\cdot) \n_{L^2} \n (v-w)(t,\cdot) \n_{L^2 }+ \n v(t,\cdot) \n_{L^2} \n \partial_x (v-w)(t,\cdot) \n_{L^2},  
\end{equation*}
then
\begin{eqnarray*}
t^{1/2} \n (v \partial_x v-w\partial_x w)(t,\cdot) \n_{L^{1}} &\leq & \n (v-w)(t,\cdot) \n_{L^2} |||w||| + |||v||| t^{1/2} \n \partial_x (v-w)(t,\cdot) \n_{L^2}, \\
&\leq & 2 R |||v-w|||.   
\end{eqnarray*}
Therefore, we obtain
\begin{eqnarray} \label{estim3}
\n \partial_{x}(\Theta v-\Theta w)(t,\cdot) \n_{L^2(\R)}  &\leq& 2 K_0 I t^{-1/4}R |||v-w||| + K_1 J ||u_\phi ||_{C^1_b (\R)} T^{1/2} |||v-w|||\nonumber \\ 
&+&   K_1 J ||u_\phi||_{C^1_b (\R)}|||v-w|||.     
\end{eqnarray}
Finally, using \eqref{estim2} and \eqref{estim3}, we get
\begin{eqnarray*}
|||\Theta v -\Theta w |||& \leq& [(2 + I)2R K_0 T^{1/4}+ (2+J) \n u_\phi \n_{C^1_b (\R)} K_1 T^{1/2} \\
&\,& \hspace{1 cm} + \, K_1 J T ||u_\phi ||_{C^1_b (\R)}]|||v-w|||. \\
\end{eqnarray*}

\noindent \textbf{Last step: conclusion.}
For any $T_\ast>0$ sufficiently small such that \eqref{Tsmall} holds true and $$(2+I)2 R K_0 T_\ast^{1/4}+ (2+J) \n u_\phi \n_{C^1_b (\R)} K_1 T_\ast^{1/2}+K_1 J T_\ast ||u_\phi ||_{C^1_b (\R)}<1,$$ $\Theta$ is a contraction from $B_R$ into itself. The Banach fixed point theorem then implies that $\Theta$ admits a unique fixed point $v \in C([0,T_\ast]; L^{2}(\R)) \cap C((0,T_\ast]; H^{1}(\R))$ which is a mild solution to \eqref{fowlermodif2}.  

\end{proof}

\begin{lemme}[Regularity $H^2$ of $v(t, \cdot) $]
Let $v_0 \in L^2(\R)$ and $\phi \in C^2_b(\R)$. There exists $T_\ast'>0$ that only depends on
$||v_0||_{L^2(\R)}$ and $\n u_\phi \n_{C_{b}^{2}(\R)}$ such that \eqref{fowlermodif2} admits a unique
mild solution $v \in C([0,T_\ast']; L^{2}(\R)) \cap C((0,T_\ast']; H^{2}(\R)) .$ 
Moreover, $v$ satisfies 
\begin{equation*}
\sup_{t \in (0,T'_*]} t^{1/2} ||\partial_{x} v(t, \cdot) ||_{L^2(\R)} < + \infty \hspace{0.3 cm} \mbox{ and } \hspace{0.3 cm} \sup_{t \in (0,T'_*]} t ||\partial^2_{xx} v(t, \cdot) ||_{L^2(\R)} < + \infty.
\end{equation*}
\label{regularityH2}
\end{lemme}

\begin{proof}
To prove this result, we use again a contracting fixed point theorem. But this time, it is the gradient of the solution $v$ which is searched as a fixed point.  \\
From Proposition \ref{existlocal}, there exists $T_\ast>0$ which depends on $||v_0||_{L^{2}(\R)}$ and $||u_\phi||_{C^{1}(\R)}$ such that $v \in C([0,T_\ast]; L^{2}(\R)) \cap C((0,T_\ast]; H^{1}(\R)) $ is a mild solution to \eqref{fowlermodif2}. Since $v \in C((0,T_\ast]; H^{1}(\R))$,
 we can consider the gradient of $v(t, \cdot)$ for any $t \in (0, T_*]$. Let then 
$t_0 \in (0,T_*)$ and $T^{'}_* \in (0,T_*-t_0]$. We consider the same complete metric space $X$ defined in the proof of Proposition \ref{existlocal} and we take the norm $\n \vert \cdot \n \vert$ defined in \eqref{norme}:
\begin{equation*}
X:= \left\{w \in C([0,T^{'}_*];L^2(\R)) \cap C((0,T^{'}_*];H^1(\R))
\mbox{ s.t. $w(0,\cdot)=w_0$ and $|||w||| <+\infty$} \right\},
\end{equation*} 
with the initial data $w_0 = \partial_x v(t_0, \cdot) $. \\
We now wish to apply the fixed point theorem at the following function
\begin{eqnarray*}
\Theta w: t \in [0,T^{'}_*] & \rightarrow & K(t,\cdot) \ast w_0 -
\int_{0}^{t}
\partial_x K (t-s,\cdot) \ast  \left( \bar{v} w \right) (s,\cdot)ds \\
& \, & \hspace{1 cm} - \int_{0}^{t} \partial_x K (t-s,\cdot) \ast \left( \partial_x( u_{\phi})  \bar{v} \right)(s, \cdot) \, ds \\ 
& \, & \hspace{1 cm} - \int_{0}^{t} \partial_x K (t-s,\cdot) \ast \left( u_{\phi} w \right) (s, \cdot) \, ds  , 
\end{eqnarray*}
where $\bar{v}(t, \cdot) := v(t_0 + t, \cdot) $. 
First, we leave to the reader to verify that $\Theta $ maps $X$ into itself. 
The proof is similar to the one given in Proposition \ref{existlocal}.  \\
\noindent 
For any $w \in X$, we have from Young inequalities and Remark \ref{remarkestimation}
\begin{eqnarray*}
|| \Theta w(t, \cdot) ||_{L^{2}(\R)} & \leq & e^{\alpha_0 T'_*} ||w_0||_{L^{2}(\R)} + ||\bar{v}||_{C \left([t_0, T^{'}_*]; H^{1}(\R) \right) }||| w ||| \int_{0}^{t} ||\partial_x K(t-s, \cdot)||_{L^{2}(\R)} ds \\
& \, & \hspace{1 cm} +  \, ||u_ {\phi} ||_{C^1_b(\R)} || \bar{v} ||_{C \left([t_0, T^{'}_*]; H^{1}(\R) \right) } \int_{0}^{t} ||\partial_x K(t-s, \cdot)||_{L^{1}(\R)} ds \\
& \, & \hspace{1 cm} + \, ||u_ {\phi} ||_{C^1_b(\R)} ||| w ||| \int_{0}^{t} ||\partial_x K(t-s, \cdot)||_{L^{1}(\R)} ds, 
\end{eqnarray*}
and from Proposition \ref{kernel20}, we get
\begin{eqnarray}
|| \Theta w(t, \cdot) ||_{L^{2}} & \leq & e^{\alpha_0 T'_*} ||w_0||_{L^{2}} + 4 K_0 T'^{1/4}_* ||\bar{v}||_{C \left([t_0, T^{'}_*]; H^{1} \right) } ||| w ||| \hspace{1 cm} \nonumber \\
&+ &2 K_1 T'^{1/2}_* ||u_ {\phi} ||_{C^1_b} || \bar{v} ||_{C \left([t_0, T^{'}_*]; H^{1} \right) } +  2 K_1 T'^{1/2}_* ||u_ {\phi} ||_{C^1_b} ||| w |||. 
\label{estim1H2}
\end{eqnarray}
Differentiating $\Theta v(t, \cdot)$ w.r.t the space variable, we obtain
\begin{eqnarray*}
\partial_x \Theta v(t, \cdot) & = & \partial_x K(t, \cdot) \ast w_0 - \int_0^t \partial_x K(t-s, \cdot) \ast \partial_x ( \bar{v} \, w)(s, \cdot) \, ds \\ 
&\,& - \int_0^t \partial_x K(t-s, \cdot) \ast \partial_x ( \partial_x (u_{\phi}) \, \bar{v})(s, \cdot) \, ds - \int_0^t \partial_x K(t-s, \cdot) \ast \partial_x (  u_{\phi} \, w)(s, \cdot) \, ds, \\
\end{eqnarray*}
and developing, we get 
\begin{eqnarray*}
\partial_x \Theta v(t, \cdot) & = &  \partial_x K(t, \cdot) \ast w_0 - \int_0^t \partial_x K(t-s, \cdot) \ast \left[ w \, \partial_x \bar{v} + \bar{v} \, \partial_x w  \right](s, \cdot)  \, ds \\ 
&\,&  - \int_0^t \partial_x K(t-s, \cdot) \ast \left[ \partial_x^2(u_\phi) \, \bar{v} + \partial_x(u_\phi) \, \partial_x \bar{v}\right] (s, \cdot) \, ds \\
&\,& - \int_0^t \partial_x K(t-s, \cdot) \ast \left[ \partial_x(u_\phi) \, w + u_\phi \, \partial_x w \right](s, \cdot) \, ds. 
\end{eqnarray*} 
Now, from Young inequalities, we have
\begin{eqnarray*}
||\partial_x \Theta v(t, \cdot)||_{L^2} &\leq &  ||\partial_x K(t, \cdot)||_{L^1} ||w_0||_{L^2} + \int_0^t || \partial_x K(t-s, \cdot)||_{L^2}   ||w \, \partial_x \bar{v} (s, \cdot) ||_{L^1} \, ds \\
&\,&  +  \int_0^t || \partial_x K(t-s, \cdot)||_{L^2}  ||\bar{v} \, \partial_x w (s, \cdot)||_{L^1}  \, ds \\
&\,& + \int_0^t ||\partial_x K(t-s, \cdot)||_{L^1} \left[ || \partial_x^2(u_\phi) \, \bar{v} (s, \cdot)||_{L^2} + ||  \partial_x(u_\phi) \, \partial_x \bar{v}(s, \cdot)||_{L^2} \right]  \, ds \\
&\,& + \int_0^t ||\partial_x K(t-s, \cdot)||_{L^1} \left[ || \partial_x(u_\phi) \, w (s, \cdot)||_{L^2} + ||u_\phi \, \partial_x w(s, \cdot) ||_{L^2} \right]  \, ds.
\end{eqnarray*}
Finally, from Proposition \ref{kernel20}, we obtain
\begin{eqnarray*}
||\partial_x \Theta v(t, \cdot)||_{L^2} &\leq &  t^{-1/2} K_1 ||w_0||_{L^2} + 4 \, t^{1/4} \, K_0  || \bar{v}||_{C([t_0; T^{'}_*]; H^{1})} |||w ||| \\
&\,&  + \int_0^t K_0 \, (t-s)^{-3/4} \, s^{-1/2} \, ds \, || \bar{v}||_{C([t_0; T^{'}_*]; H^{1})} \sup_{s \in (0, T'_*]} s^{1/2} ||\partial_x w (s, \cdot)||_{L^2}  \\
&\,& + 4 \, K_1 \, t^{1/2} \, ||u_\phi ||_{C^2_b} || \bar{v}||_{C([t_0; T^{'}_*]; H^{1})} + 2 \, K_1 \, t^{1/2} \, ||u_\phi ||_{C^1_b} ||| w ||| \\
&\,& + \int_0^t K_1 \, (t-s)^{-1/2} \, s^{-1/2} \, ds \, ||u_\phi||_{C^2_b} \sup_{s\in (0, T'_*]  } s^{1/2} ||\partial_x w (s, \cdot)||_{L^2}.  
\end{eqnarray*}
In other words, we have for all $t \in (0,T'_*]$
\begin{eqnarray}
t^{1/2} \, ||\partial_x \Theta v(t, \cdot)||_{L^2} &\leq &  K_1 ||w_0||_{L^2} + 4 \, T_{*}^{'3/4} \, K_0  || \bar{v}||_{C([t_0; T^{'}_*]; H^{1})} |||w ||| \nonumber \\
&\,&  +  K_0  \, I \,  T_{*}^{'1/4} \, || \bar{v}||_{C([t_0; T^{'}_*]; H^{1})} \, |||w ||| +  4 \, K_1 \,  T_{*}^{'}\, ||u_\phi ||_{C^2_b} || \bar{v}||_{C([t_0; T^{'}_*]; H^{1})} \nonumber \\
&\,&  + 2 \, K_1 \,  \pi \,  T_{*}^{'} \, ||u_\phi ||_{C^1_b} ||| w ||| +  K_1 \, T_{*}^{'1/2} \, ||u_\phi||_{C^2_b} \, |||w |||, 
\label{estim2H2}
\end{eqnarray}
where $I = B(\frac{1}{2},\frac{1}{4})$. 
Hence, using \eqref{estim1H2} and  \eqref{estim2H2}, we get
\begin{eqnarray*}
|||\Theta w ||| &\leq & e^{\alpha_0 T'_*} ||w_0||_{L^{2}(\R)} + K_1 ||w_0 ||_{L^{2}(\R)}  + 2 K_1 ||u_{\phi}||_{C^{2}_{b}(\R)}  || \bar{v}||_{C([t_0; T^{'}_*]; H^{1}(\R))}(2 \, T'_* + T'^{1/2}_* ) \\
&\, & \hspace{0.5 cm} + \, C |||w||| (T^{'1/4}_* + T^{'1/2}_*+T^{'3/4}_*+ T^{'}_* ),
\end{eqnarray*}
for some positive constant $C$ which depends on $K_0, K_1, || \bar{v}||_{C([t_0; T^{'}_*]; H^{1}(\R))}$ and $||u_{\phi}||_{C^{2}_{b}(\R)} $. \\
We next leave to reader to verify that: for any $w_1, w_2 \in X$,
\begin{eqnarray*}
|||\Theta w_1 - \Theta w_2 ||| &\leq & C'(T'^{1/4}_* + T'^{1/2}_* +T'^{3/4}_*+ T'_* ) |||w_1- w_2|||,
\end{eqnarray*}
where $C'$ is a positive constant which depends on $K_0, K_1, || \bar{v}||_{C([t_0; T^{'}_*]; H^{1}(\R))}$ and $||u_{\phi}||_{C^{2}_{b}(\R)} $. \\
Then, if $T^{'}_* > 0$ satisfies 
\begin{eqnarray*}
&\, & e^{\alpha_0 T'_*} ||w_0||_{L^{2}(\R)} + K_1 ||w_0 ||_{L^{2}(\R)}  + 2 K_1 ||u_{\phi}||_{C^{1}_{b}(\R)}  || \bar{v}||_{C([t_0; T^{'}]; H^{1}(\R))}(2 T'_* + T^{'1/2}_* ) \\
&\, & \hspace{6.5 cm } + \, C \, R \, (T^{'1/4}_* + T^{'1/2}_*+  T^{'3/4}_* + T^{'}_* ) \leq R,
\end{eqnarray*}
and 
\begin{equation*}
C'\, (T^{'1/4}_* + T^{'1/2}_* + T^{'3/4}_* + T_* ) < 1, 
\end{equation*}
$\Theta : B_{R}(X) \longrightarrow B_{R}(X)$ is a contraction, where $B_{R}(X)$ is ball of $X$ of radius $R$ centered at the origin. 
Using a contracting point fixed theorem, it exists a unique fixed point, which we denote by $w$. 
But it is easy to see that $\Theta \partial_x \bar{v} = \partial_x \bar{v}$ taking into account the space derivated of the Duhamel formulation \eqref{duhamel2}. Thanks to a uniqueness argument, we deduce that $ w = \partial_x \bar{v}$ and thus that $v \in C((0, T'_*]; H^{2}(\R))$, which completes the proof of this lemma.
\end{proof}

Let us now prove the global-in-time existence of mild solution $v$.

\begin{proposition}[Global-in-time existence \label{global}]
Let $v_0 \in L^2(\R)$, $\phi \in C^2_b(\R)$ and $T>0$. Then, there exists a (unique) mild solution $v \in C([0,T]; L^{2}(\R))\cap C((0,T]; H^{2}(\R))$ to \eqref{fowlermodif2}. Moreover, $v$ 
satisfies the PDE \eqref{fowlermodif2} in the distribution sense. 
\end{proposition} 

\begin{proof}
\noindent \textbf{First step: $v$ is a distribution solution.} Taking the Fourier transform w.r.t the space variable in \eqref{duhamel2}, we get for all $t \in [0,T]$ and all $\xi \in \R$,
\begin{eqnarray} \label{equal1}
\F(v(t,\cdot))(\xi)&=&e^{-t \psi_{\I}(\xi)} \F v_0(\xi)- \int_{0}^{t} i \pi \xi e^{-(t-s) \psi_{\I}(\xi)} \F(v^2(s,\cdot))(\xi) ds \nonumber \\
&\,& - \int_{0}^{t} 2 i \pi \xi e^{-(t-s) \psi_{\I}(\xi)} \F(u_\phi v(s,\cdot))(\xi) \, ds. 
\end{eqnarray} 
Define
$$
G(t,\xi)=- \int_{0}^{t} 2 i \pi \xi e^{-(t-s) \psi_{\I}(\xi)} \F\left( \frac{v^2}{2} + u_\phi v\right) (s,\cdot)(\xi) \, ds.
$$
Classical results on ODE imply that $G$ is differentiable w.r.t the time with \\ 
\begin{eqnarray}\label{equal2}
\partial_t G(t,\xi)+ \psi_{\I}(\xi) G(t,\xi) &=& -i \pi \xi \F\left( v^2(t, \cdot)\right)(\xi) - 2i \pi \xi \F\left( (u_\phi v)(t, \cdot) \right)(\xi), \nonumber   \\
&=&-\F \left(  \partial_x(\frac{v^2}{2})(t,\cdot) \right) (\xi)-\F \left( \partial_x(u_\phi v)(t,\cdot) \right) (\xi).
\end{eqnarray}
Let us now prove that all terms in \eqref{equal2} are continuous with values in $L^2$. 
Since, $v \in C((0,T]; H^{1}(\R))$ then $\partial_x(v^2), \, \partial_x(u_\phi v) \in C((0,T]; L^{2}(\R))$. We thus deduce that $\F \left( \partial_x(\frac{v^2}{2}) \right) $ and $\F \left( \partial_x (u_\phi v) \right) $ are continuous with values in $L^2(\R)$. 
Moreover,  Equation \eqref{equal1} implies that 
\begin{eqnarray*}
\psi_\I G(t, \cdot) = \psi_\I \left( \F(v(t,\cdot)) - e^{-t \psi_\I } \F v_0  \right), 
\end{eqnarray*}
and so $\psi_\I \, G(t, \cdot) $ is continuous with values in $L^2$. Indeed, \\
\begin{eqnarray*}
\int_\R | \psi_\I (\xi) G(t, \xi ) |^2 \, d\xi & = & \int_{-1}^{1} |\psi_\I (\xi) G(t,\xi) |^2 \, d\xi +  \int_{\R \setminus (-1,1)} |\psi (\xi) G(t,\xi) |^2 \, d\xi, \\
&\leq & \sup_{\xi \in [-1,1]} | \psi_\I (\xi) |^2 ||G(t,\cdot) ||^2_{L^2(\R)} + C  \int_{\R \setminus (-1,1)} |\xi^2 G(t,\xi) |^2 \, d\xi, \\
&\leq &  \sup_{\xi \in [-1,1]} | \psi_\I (\xi) |^2 ||G(t,\cdot) ||^2_{L^2(\R)} \\
&\,& \hspace{2.2 cm} + \, C \, \int_{\R \setminus (-1,1)} | \F( \partial^2_{xx} v(t,\cdot)) - \xi ^2 e^{-t \psi_\I(\xi)} \F v_0 |^2 \, d\xi, \\
&\leq & \sup_{\xi \in [-1,1]} | \psi_\I (\xi) |^2 ||G(t,\cdot) ||^2_{L^2(\R)} + \tilde{C} || v(t,\cdot) ||^2_{H^{2}(\R)} \\
&\,& \hspace{2.5 cm}  + \tilde{C}|| v_0||_{L^{2}(\R)}^2 + C ||v(t, \cdot)||_{H^2}||v_0||_{L^2} , \\
&< & \infty,
\end{eqnarray*}
because $\psi_\I$ behaves at infinity as $|\cdot|^2$. $C, \tilde{C}$ are two positive constants. Hence, we have that the function $t \rightarrow \psi_\I G(t, \cdot) \in L^{2}(\R, \mathbb{C})$ is continuous. Finally, we have proved that all the terms in \eqref{equal2} are continuous with values in $L^2$. 
Therefore, from \eqref{equal2}, we get that $G \in C^{1}([0,T];L^2(\R,\mathbb{C})) $ and then
\begin{equation*}
\frac{d}{dt}(G(t, \cdot))+ \psi_{\I} \, G(t, \cdot)=-\F \left( \partial_x(\frac{v^2}{2})(t,\cdot) \right) -\F \left( \partial_x(u_\phi v)(t,\cdot)\right) .
\end{equation*}
Moreover,  $t \in [0,T]\rightarrow e^{-t \psi_{\I}} \F v_0 \in L^{2}(\R,\mathbb{C})$ is $C^1$ with
\begin{equation*}
\frac{d}{dt}(e^{-t \psi_{\I}} \F v_0)+ \psi_{\I} e^{-t \psi_{\I}} \F v_0=0.
\end{equation*}
From Equation \eqref{equal1}, we infer that $\F v$ is $C^1$ on $[0,T]$ with values in $L^2$ with 
\begin{equation*}
\frac{d}{dt}\F (v(t, \cdot))=-\psi_{\I} \F(v(t,\cdot))-\F \left( \partial_x(\frac{v^2}{2})(t,\cdot)\right) -\F \left( \partial_x(u_\phi v)(t, \cdot)\right) .
\end{equation*} 
Since $\F$ is an isometry of $L^2$, we deduce that $v \in C^{1}([0,T]; L^{2}(\R))$  and by \eqref{fourier2}, we get
\begin{eqnarray}
\frac{d}{dt}(v(t,\cdot))&=&-\F^{-1}(\psi_{\I} \F \left( v(t,\cdot))\right) -\partial_x(\frac{v^2}{2})(t,\cdot)-\partial_x(u_\phi v)(t,\cdot), \nonumber \\
&=& -\I[v(t,\cdot)]+ \partial_{xx}^2 v(t,\cdot)-\partial_x(\frac{v^2}{2})(t,\cdot)-\partial_x(u_\phi v)(t,\cdot).
\label{egalitedistribution}
\end{eqnarray}
We are now going to prove that $v$ satisfies the PDE \eqref{fowlermodif2} in the distribution sense. Let us note 
\begin{equation*}
w(t, \cdot):=-\I[v(t,\cdot)]+ \partial_{xx}^2 v(t,\cdot)-\partial_x(\frac{v^2}{2})(t,\cdot)-\partial_x(u_\phi v)(t,\cdot)
\end{equation*}
and let us show that 
\begin{equation*}
\partial_t v= w \hspace{0.5 cm} \mbox{in} \hspace{0.5 cm} \mathcal{D}'((0,T) \times \R ).
\end{equation*}
By definition, we have for any $\varphi \in \mathcal{D}(0,T)$ and $\psi \in \mathcal{D}(\R)$:
\begin{eqnarray*}
< \partial_t v, \varphi \psi > &=& - \int_{0}^{T} \int_{\R} v(t,x) \frac{d \varphi}{d t} \psi(x) \, dt \, dx, \\
&=& - \int_{0}^{T} \left( \int_\R v(t,x) \psi(x) \, dx \right) \frac{d \varphi(t)}{d t} \, dt  .
\end{eqnarray*}
Therefore, it is enough to prove that
\begin{equation*}
\int_0^T \left( \int_\R w(t,x) \psi (x) dx \right)  \varphi(t) dt = - \int_0^T \left( \int_\R v(t,x) \psi (x) dx \right)  \varphi '(t) dt,
\end{equation*} 
i.e.
\begin{equation*}
 \frac{d}{dt} \int_\R v(t,x) \psi (x) dx = \int_\R w(t,x) \psi (x) \, dx, 
\end{equation*}
in the sense of $\mathcal{D}^{'}(0,T)$.
But by \eqref{egalitedistribution}, we have that the function 
$$t \in (0,T) \longmapsto \int_\R v(t,x) \psi (x) dx \in \R $$
is $C^{1}$ and 
\begin{equation*}
\frac{d}{dt} \int_\R v(t,x) \psi (x) dx = \int_\R w(t,x) \psi (x) dx
\end{equation*}
in the classical sense, which proves that the mild solution $v$ is a distribution solution of \eqref{fowlermodif2}. \\ 

\noindent \textbf{Second step: A priori estimate.} By the first step, we have
$$\partial_t v+ \partial_x (\frac{v^2}{2}+ u_\phi v)+ \I[v]-\partial_{xx}^{2} v=0$$
 in the distribution sense. Therefore, multiplying this equality by $v$ and integrating w.r.t the space variable, we get: 
\begin{equation} \label{relation10}
\int_{\R} v_{t}v \: dx+ \int_{\R} \left(\I[v]-v_{xx}\right)v \: dx +\int_{\R} \left(u_\phi v\right)_{x} v \: dx=0
\end{equation}
because the nonlinear term is zero. Indeed, integrating by parts, we have
\begin{equation*}
\int_{\R} \partial_x (\frac{v^2}{2}) v \, dx= -\int_{\R} \frac{v^2}{2} \partial_x v \, dx= - \frac{1}{2} \int_{\R} \partial_x (\frac{v^2}{2}) v \, dx.  
\end{equation*}
There is no boundary term from the infinity because for all $t \in (0,T]$, $v(t,\cdot) \in H^{2}(\R)$.
Using \eqref{fourier2} and the fact that $\int_\R (\I[v]-\partial_{xx}^2 v) v \, dx$ is real, we get
\begin{equation} \label{relation20}
\int_\R (\I[v]-\partial_{xx}^2 v) v \, dx = \int_\R \F^{-1}(\psi_\I
\F v) v \, dx = \int_\R \psi_\I |\F v|^2 d\xi
 = \int_\R \mbox{Re}(\psi_\I) |\F v|^2 d\xi.
\end{equation} 
Moreover, since $u_{\phi} v \in H^{1}(\R)$ we have 
\begin{equation} \label{relation30}
\int_{\R} \left(u_\phi v\right)_{x} v \: dx=- \int_{\R} u_\phi v v_{x} \: dx= - \int_{\R} u_\phi \left(\frac{v^{2}}{2}\right)_{x} \: dx= \int_{\R} (\partial_{x} u_{\phi}) \frac{v^{2}}{2} \: dx.
\end{equation}
Using \eqref{relation10}, \eqref{relation20} and \eqref{relation30}, we obtain
\begin{eqnarray*}
\frac{1}{2} \frac{d}{dt}\n v(t, \cdot) \n_{L^2}^{2} & \leq & (\alpha_{0}+ C_\phi) \n v(t, \cdot) \n_{L^{2}}^{2} 
\end{eqnarray*}
where $\alpha_0=-\min \mbox{Re} \left( \psi_{\I}\right) >0$ and $C_\phi = \frac{1}{2}\n u_\phi \n_{C^{1}_b } $. Finally, we get for all $t \in [0,T]$ the following estimate
\begin{equation} \label{energy2}
\n v(t, \cdot) \n_{L^2(\R)}\leq e^{(\alpha_{0}+ C_\phi) t} \n v_{0} \n_{L^2(\R)}.
\end{equation}
\\
\noindent \textbf{Last step: global-in-time existence.} Up to this point, we know thanks to Proposition \ref{existlocal} and Lemma \ref{regularityH2} that there exists $T_\ast =T_\ast(\n v_{0} \n_{L^{2}(\R)}, \n u_\phi \n_{C^{2}_b(\R)} ) >0$  such that $v \in C([0,T_\ast]; L^2(\R)) \cap C((0,T_\ast]; H^2(\R))$ is a mild solution of \eqref{fowlermodif2} on $(0,T_{\ast}]$. Let us define 
$$t_{0}:=\sup\left\{t>0 \hspace{0.1 cm}/ \mbox{ there exists a mild solution of \eqref{fowlermodif2} on } \left(0,t\right) \mbox{ with initial condition } v_{0} \right\}.$$
To prove the global-in-time existence of a mild solution, we have to prove that $t_0 \geq T$, where $T$ is any positive constant. Assume by contradiction that $t_0 < T$. 
With again the help of Proposition \ref{existlocal}, there exists $T'_\ast > 0 $ such that for any initial data $w_0$ that satisfy
\begin{equation} \label{ci}
\n w_0 \n_{L^2(\R)} \leq e^{(\alpha_0 + C_\phi) t_0} \n v_0 \n_{L^2(\R)},
\end{equation}  
it exists a mild solution $w$ on $(0,T'_\ast]$. 
Using \eqref{energy2}, we have that $w_0:=v(t_0-T_\ast^{'}/2, \cdot)$ satisfies \eqref{ci}. Therefore, using an argument of uniqueness, we deduce that \\
$ v (t_0-T_\ast^{'}/2 + t, \cdot) = w(t,\cdot)$ for all $t \in [0,T_\ast^{'}/2)$.
To finish with, we define $\tilde{v}$ by $\tilde{v}=v$ on $[0,t_0)$ and $\tilde{v}(t_0-T_{\ast}^{'}/2+t, \cdot)=w(t,\cdot)$ for $t \in [T^{'}_\ast/2, T_\ast^{'}]$. Hence, $\tilde{v}$ is a mild solution on $[0,t_0+ T_\ast^{'}/2]$ with initial datum $v_0$, which gives us a contradiction.

\end{proof}

\section{Regularity of the solution \label{regularity}}

This section is devoted to the proof of the existence of classical solutions $v$ to \eqref{fowlermodif2}.

\begin{proposition}[Solution in the classical sense]
Let $v_0 \in L^2(\R)$, $\phi \in C^2_b(\R)$ and $T > 0$. The unique mild solution $v \in C([0,T]; L^{2}(\R))\cap C((0,T]; H^{2}(\R))$ of \eqref{fowlermodif2} belongs to $ C^{1,2} \left( (0,T] \times \R \right) $ and 
satisfies $$\partial_t v + \partial_x{ \left( \frac{v^2}{2} + u_{\phi} v \right) } + \I[v] - \partial^2_{xx} v = 0,$$ on $(0,T] \times \R$ in the classical sense.  \\
\label{classicalsense}
\end{proposition}

\begin{proof}

\noindent \textbf{First step: $C^2$-regularity in space.} Let us take any $t_0 \in (0,T]$ as initial time and let $T' \in (0, T-t_0]$. 
Differentiating the Duhamel formulation \eqref{duhamel2} two times w.r.t the space, we get for any $t \in [0, T']$,
\begin{eqnarray*}
\partial_{xx}^2 v(t + t_0, \cdot) &=& K(t, \cdot) \ast \partial_{xx}^2 v(t_0, \cdot) - \int_{0}^{t} \partial_{x} K(t-s, \cdot) \ast \left( u_1 + u_2\right)(t_0 + s, \cdot) \, ds ,
\end{eqnarray*}
where $u_1 : = (\partial_x v)^2 + v \partial^2_{xx}v$ and $u_2 := v \, \partial_x^2 u_\phi + 2\,  \partial_x u_\phi \, \partial_x v + u_\phi \, \partial_{xx}^2 v$. 
Since $v \in C((0,T] ; H^2(\R))$ then $u_2 \in C\left((0,T] ; L^{2}(\R) \right)$ and 
from the Sobolev embedding $H^2(\R) \hookrightarrow C^{1}_{b}(\R)$,  we get that $u_1 \in C((0,T]; L^1(\R)\cap L^{2}(\R) )$.  
Let us now define the following functions
\begin{equation*}
F_i(t, x) : = \int_{0}^{t} \partial_x K(t-s, \cdot) \ast u_i(t_0 + s, \cdot)(x) \, ds, \hspace{1 cm} \mbox{ for } i=1,2. 
\end{equation*}

For all $x,y \in \R$, we have thanks to Cauchy-Schwartz inequality
\begin{eqnarray*}
&\, & |\partial_x K(t-s, \cdot) \ast  u_i(t_0 + s, \cdot)(x) - \partial_x K(t-s, \cdot) \ast u_i(t_0 + s, \cdot)(y)| \\
& \, & \hspace{4 cm} \leq \int_{\R}  | \partial_x K(t-s, z)| \, |u_i(t_0 + s, x-z ) - u_i(t_0 + s, y-z )| \, dz, \\
&\, & \hspace{4 cm} \leq ||\mathcal{T}_{(x-y)}\left(u_{i}(t_0+s, \cdot) \right) - u_i(s+t_0, \cdot) ||_{L^{2}(\R)} ||\partial_x K(t-s, \cdot) ||_{L^{2}(\R)},
\end{eqnarray*}
where $ \mathcal{T}_{z} \varphi$ denotes the translated function $x \rightarrow \varphi(x + z)$. \\
Therefore, for all $x,y \in \R$ and all $t \in [0,T']$, we deduce that 
\begin{eqnarray}
|F_i(t, x) - F_i(t, y) | &\leq & \int_{0}^{t} K_0 (t-s)^{-3/4} ||\mathcal{T}_{(x-y)}\left(u_{i}(t_0+s, \cdot) \right) - u_i(t_0 + s, \cdot) ||_{L^{2}(\R)} \, ds, \nonumber \\
&\leq& 4 K_0 T'^{1/4} \sup_{s \in [0,T']} ||\mathcal{T}_{(x-y)}\left(\bar{u}_{i}(s, \cdot) \right) - \bar{u}_i(s, \cdot) ||_{L^{2}(\R)},
\label{continuity}
\end{eqnarray}
where $\bar{u}_i(s, \cdot) = u_i(t_0+ s, \cdot)$. Then, $\bar{u}_i$ is uniformly continuous with values in $L^2$ as a continuous function on a compact set $[0, T']$. Therefore, for any $\epsilon > 0$, there exists a finite sequence $ 0 = s_0 < s_1 < \cdots < s_N = T'$ such that for any $s \in [0, T']$, there exists $j \in \left\lbrace 0, \cdots, N-1 \right\rbrace $ such that
\begin{equation*}
||\bar{u}_i(s, \cdot ) - \bar{u}_i(s_j, \cdot ) ||_{L^2(\R)} \leq \epsilon.
\end{equation*}
Therefore, using \eqref{continuity} we have
\begin{eqnarray*}
|F_i(t, x) - F_i(t, y) | 
&\leq& 4 K_0 T'^{1/4}  \sup_{s \in [0,T']}  ||\mathcal{T}_{(x-y)}\left(\bar{u}_{i}(s, \cdot) \right) -\mathcal{T}_{(x-y)}\left(\bar{u}_{i}(s_j, \cdot) \right)||_{L^{2}}  \\ 
&+& 4 K_0 T'^{1/4} \left\lbrace  ||\mathcal{T}_{(x-y)}\left(\bar{u}_{i}(s_j, \cdot) \right) - \bar{u}_{i}(s_j, \cdot)||_{L^{2}} + \sup_{s \in [0,T]} ||\bar{u}_{i}(s, \cdot) - \bar{u}_{i}(s_j, \cdot)||_{L^{2}} \right\rbrace  .
\end{eqnarray*} 
And since $||\mathcal{T}_{(x-y)}\left(\bar{u}_{i}(s, \cdot) \right) -\mathcal{T}_{(x-y)}\left(\bar{u}_{i}(s_j, \cdot) \right)||_{L^{2}(\R)} = ||\bar{u}_{i}(s, \cdot) - \bar{u}_{i}(s_j, \cdot)||_{L^{2}(\R)} $, we get
\begin{eqnarray*}
|F_i(t, x) - F_i(t, y) | 
&\leq& 4 K_0 T'^{1/4}  \left\lbrace ||\mathcal{T}_{(x-y)}\left(\bar{u}_{i}(s_j, \cdot) \right) - \bar{u}_{i}(s_j, \cdot)||_{L^{2}(\R)} + 2 \epsilon \right\rbrace .
\end{eqnarray*} 
And since the translated function is continuous in $L^2(\R)$, we have
$$||\mathcal{T}_{(x-y)}\left(\bar{u}_{i}(s_j, \cdot) \right) - \bar{u}_{i}(s_j, \cdot)||_{L^{2}(\R)} \rightarrow 0, $$ as $(x-y)\rightarrow 0 $. 
Hence, 
\begin{equation*}
\limsup_{(x-y)\rightarrow 0} |F_i(t, x) - F_i(t, y) |  \leq 2 \epsilon.
\end{equation*}
Taking the infimum w.r.t $\epsilon >0$, we infer that $F_i$ is continuous w.r.t the variable $x$. 
Moreover, arguing as the proof of Proposition \ref{welldefined}, we get that $F_i \in C\left([0, T']; L^{2}(\R) \right)$. From classical results, we then deduce that $F_i$ is continuous w.r.t the couple $(t,x)$ on $[0,T'] \times \R$. \\
Moreover, since $v(t_0, \cdot) \in H^{2}(\R)$, we can easily check that $(t,x) \rightarrow K(t, \cdot) \ast \partial^2_{xx} v(t_0, \cdot)(x)$ is continuous on $(0,T] \times \R$. Finally, we get that $ \partial_{xx}^2 v \in C \left([t_0, T] \times \R \right)$ and since $t_0$ is arbitrary in $(0,T]$, we conclude that $\partial_{xx}^2 v \in C \left((0, T] \times \R \right)$. \\

\noindent \textbf{Second step: $C^1$-regularity in time.} From Proposition \ref{global}, we know that 
the terms $\partial_t v$ and $ - \partial_{x}\left(\frac{v^2}{2} + u_{\phi} v \right) + \partial_{xx}^2 v - \I[v] $ have the same regularity. Moreover, by the first step of this proposition, we have that $\partial_{xx}^2 v \in C((0,T] \times \R)$ and from Sobolev embeddings and Remark  \ref{remarkfourier}, we deduce that $ \partial_{x}\left(\frac{v^2}{2} + u_{\phi} v \right)$ and $\I[v] $ belong to $ C((0,T] \times \R) $.
Finally, we obtain that $\partial_t v \in C((0,T] \times \R)$ and thus $v \in C^{1,2}((0,T] \times \R)$. 
The proof of this Proposition is now complete.
\end{proof}



\bigskip 

\noindent \textbf{Acknowledgements.} The author would like to thank Pascal Azerad for helpful discussions around this work. This work is part of ANR project Mathocean (ANR-08-BLAN-0301-02).

\bibliographystyle{plain}

\end{document}